\documentclass[12pt,a4paper]{article}

\usepackage{amssymb, amsmath, amsthm}

\usepackage{algorithm}
\usepackage{algpseudocode}

\usepackage[cm]{fullpage}
\usepackage[english]{babel}
\usepackage[pdftex]{graphicx}
\usepackage{booktabs}
\usepackage{xcolor}

\usepackage{algorithm}

\usepackage{enumitem}

\usepackage{hyperref}
\usepackage{autonum}

\newtheorem{thm}{Theorem}
\newtheorem{lem}{Lemma}
\newtheorem{prop}{Proposition}
\newtheorem{rem}{Remark}
\newtheorem{cor}{Corollary} 

\newtheorem{df}{Definition} 

\title{Generalized capillary‐rise models: existence and fast solvers in integral H\"older spaces}
\author{
Josefa Caballero\thanks{Departamento de Matem\'aticas, Universidad de Las Palmas de Gran Canaria, Campus de Tafira Baja, $35017$ Las Palmas de Gran Canaria, Spain.}, \and 
{\L}ukasz P{\l}ociniczak\thanks{Faculty of Pure and Applied Mathematics, Wroclaw University of Science and Technology, Wyb. Wyspia\'nskiego 27, 50-370 Wroc{\l}aw, Poland, \underline{corresponding author:} \texttt{lukasz.plociniczak@pwr.edu.pl}}, \vspace{4pt}\and
Kishin  Sadarangani$^*$
}
\date{}

\begin{document}
\maketitle

\begin{abstract}
	We study a class of nonlinear Volterra integral equations that generalize the classical capillary rise models, allowing for nonsmooth kernels and nonlinearities. To accommodate such generalities, we work in two families of function spaces: spaces with prescribed modulus of continuity and integral H\"older spaces. We establish existence results for solutions within the integral H\"older space framework. Furthermore, we analyze the behavior of linear interpolation in these spaces and provide, for the first time, sharp error estimates, demonstrating their optimality. Building on this foundation, we propose a piecewise linear collocation method tailored to solutions in integral H\"older spaces and prove its convergence. For problems admitting smoother solutions, we develop an efficient spectral collocation scheme based on Legendre nodes. Finally, several numerical experiments illustrate the theoretical results and highlight the performance of the proposed methods.
\end{abstract}



\section{Introduction}

Volterra integral equations arise naturally in the modeling of many physical, biological, and computational processes where the current state depends on a history of past values. Probably the first emergence of an integral equation was the ingenious solution of the brachistochrone problem by Abel (see an interesting account of the history, theory, and applications of integral equations in \cite{brunner2017volterra}). The Abel equations also arise in tomography \cite{deutsch1990inversion}, inverse heat conduction problems that are known to be severely ill-posed \cite{beck1996comparison,lamm2000survey}, viscoelasticity \cite{anderssen2008volterra}, and geophysics \cite{eskola2012geophysical} to name only a few important classical models. More general nonlinear Volterra equations have been used as successful tools in population modeling \cite{brauer1975nonlinear,brauer1976constant}, epidemiology \cite{guan2022global}, and fluid dynamics \cite{curle1978solution, keller1981propagation}. Furthermore, some more recent examples include the flow of fluids through exotic porous media \cite{plociniczak2018existence,lopez2024time, gerolymatou2006modelling}, financial mathematics \cite{zhu2018new}, and a very interesting new approach based on neural integral equations in which the kernel is an artificial neural network \cite{zappala2024learning}. Another well-known case is the capillary rise equation, which describes the motion of a liquid in a thin tube and involves a nonlinear memory term \cite{cai2021lucas, plociniczak2018monotonicity,okrasinska2020solvability} and can be stated as follows
\begin{equation}\label{eqn:CapillaryRiseEq}
	x(t) = \int_0^1 (1-e^{-(t-s)}) \left(1-\sqrt{2 x(s)}\right)ds
\end{equation}
We can see that although the integrand is continuous, it is not smooth - the nonlinearity belongs to the H\"older space. Similar problems have historically appeared in modeling infiltration in porous media \cite{bushell1990nonlinear}. Sometimes, the nonlinearity and a source, can lead to a blow-up phenomenon that has attracted a lot of research \cite{olmstead1995coupled}. In addition, Volterra integral equations arise very frequently in analysis of the fractional differential operator, since they are actually a combination of ordinary derivatives with integrals of a weakly singular kernel (see \cite{kilbas2006theory}). The capillary rise equation \eqref{eqn:CapillaryRiseEq} has been studied by many authors in both differential \cite{grunding2020enhanced, quere1997inertial, zhmud2000dynamics} and integral forms. For example, a series of works on the solvability, monotonicity and oscillations of the authors and their collaborators have been investigated \cite{okrasinska2020solvability, plociniczak2018monotonicity, plociniczak2021oscillatory}. Further studies included generalization to nonsmooth coefficients \cite{fricke2023analytical}, complete description of the phase plane \cite{zhang2018dynamics}, elliptical geometry of the capillary \cite{alassar2022note}, solvability in weighted L1 spaces \cite{metwali2022solvability}, and recently the full account for a general case of slip boundary conditions with a thorough theoretical analysis \cite{rapajic2025modelling}. 

In this paper, we consider a class of nonlinear Volterra equations of the form
\begin{equation}\label{eqn:MainEq0}
	x(t) = \int_0^t G(t,s) f(s, x(s)) \, ds,
\end{equation}
where the kernel $G$ and the nonlinearity $f$ are assumed to satisfy only weak regularity assumptions that provide a natural generalization of \eqref{eqn:CapillaryRiseEq}. This lack of smoothness introduces significant challenges both analytically and numerically, since classical existence and approximation results often rely on stronger regularity conditions. To analyze these equations, we work within the framework of integral H\"older spaces $J_{\alpha,\beta}[0,1]$, which provide a more flexible setting for functions with low regularity. These spaces are defined through an integral modulus of continuity, making them well suited for studying solutions that are not differentiable in the classical sense. They were introduced in \cite{appell2021holder,appell1987application} as generalizations of the classical H\"older spaces. They are strictly larger than these, as was proved in \cite{appell1987application, cichon2022banach}. In this paper, we prove the existence of solutions of \eqref{eqn:MainEq0} in these integral H\"older spaces under appropriate weak conditions on $G$ and $f$. Furthermore, we develop and analyze two numerical schemes tailored to the problem's regularity: a piecewise linear collocation method for low-regularity solutions and a spectral collocation method for smooth solutions. These methods are designed to respect the structure of the integral equation and the properties of the function space.

Volterra integral equations have been solved by many numerical methods and we refer the reader to detailed accounts in \cite{atkinson1997numerical} (a collection of various methods) and \cite{brunner2004collocation} (collocation schemes). In the case of nonsmooth solutions, the collocation scheme has been applied in \cite{riele1982collocation} to a nonlinear Volterra equation with a weakly singular kernel. Some more recent results include an application to the third-kind Volterra equation \cite{ma2023fractional} and the authors' work on collocation schemes for functional equations in the setting of H\"older spaces \cite{ okrasinska2020solvability,caballero2025collocation} (see also \cite{brunner2011analysis} for relevant results concerning functional equations). On the other hand, spectral methods have been applied to linear integral equations in \cite{tang2008spectral}. Further developments have been given in \cite{samadi2012spectral} (systems), \cite{chen2010convergence} (weakly singular kernels), \cite{sheng2014multistep, zakeri2010sinc} (nonlinear equations). See also references therein.  

The main and new results presented in this paper can be summarized as follows.
\begin{itemize}
	\item Existence theorem for \eqref{eqn:MainEq0} in the integral-H\"older space under low-regularity assumptions,
	\item piecewise linear interpolation theory in the integral-H\"older space,
	\item linear collocation scheme with the convergence proof,
	\item spectral collocation scheme with the convergence proof.
\end{itemize}
According to the authors' knowledge, our findings are the first ones in the literature that concern the above points. Moreover, devising convergent numerical schemes for approximating non-smooth solutions is usually a challenging task. For it, a careful analysis is always required, since one does not have access to the derivatives that simplify the estimation of the error. A key part of our reasoning involves deriving new interpolation error estimates in $J_{\alpha,\beta}[0,1]$. These bounds are not available in the current literature and allow us to control the maximum norm of the interpolation error using norms that naturally arise in this space (similar to Besov norms). This is essential to establish the convergence and reliability of our numerical methods. To obtain these, we borrow techniques from harmonic analysis and interpolation spaces (dyadic decomposition). For completeness, we also devise a Legendre spectral collocation scheme to solve \eqref{eqn:MainEq0} for the cases where the solution is smooth. Therefore, we present a concise framework for collocation schemes that can readily be utilized for nonsmooth and smooth solutions of our nonlinear Volterra equation. Numerical computations verify our claims and, actually, behave even better than our theoretical estimates find.  

First, we develop the existence theory for \eqref{eqn:MainEq0}. Then, in Section 3 we move to developing a numerical approach via the collocation schemes. In particular, we thoroughly treat the interpolation in integral H\"older spaces. The last section describes several numerical examples that verify our theory. 

\section{The integral equation}
In this section, we develop the theory of treating the main integral equation in a functional analytical setting. 

\subsection{Function space preliminaries}
We start by stating various definitions that will be needed for our results.
\begin{df}
	A function $\rho: \mathbb{R}_+ \mapsto \mathbb{R}_+$ is a \textbf{modulus of continuity} when it satisfies the following conditions:
	\begin{enumerate}
		\item [i)] $\rho(0) = 0$ and $\rho(\epsilon) > 0$ for all $\epsilon > 0$,
		\item [ii)] $\rho$ is nondecreasing,
		\item [iii)] $\rho(\epsilon)$ is continuous at $\epsilon = 0$.
	\end{enumerate}
\end{df}
\noindent It is also true that a given continuous function induces some modulus of continuity. However, there is a unique minimal modulus defined as follows. 
\begin{df}
	Suppose that $x: [0,1] \mapsto \mathbb{R}$. The \textbf{minimal modulus of continuity of} $\mathbf{x}$ is the function $\omega(x, \cdot): [0,1] \mapsto \mathbb{R}_+$ given by
	\begin{equation}
		\omega(x, \sigma) := \sup\left\{|x(t) - x(s)|: \, |t-s| \leq \sigma\right\} \quad \text{for any} \quad 0 \leq \sigma \leq 1.
	\end{equation}
\end{df}
\noindent Of course, a modulus of continuity of a given continuous function is indeed a modulus of continuity according to the first definition. Nonnegativity and continuity at zero are trivial, whereas monotonicity follows from the fact that taking the supremum over a larger set cannot decrease it. Therefore, both of the above definitions are compatible. We can now define the space of functions with a continuity controlled by $\rho$
\begin{equation}
	C_\rho[0,1] := \left\{x: [0,1] \mapsto \mathbb{R} \text{ such that for any } t, s\in [0,1] \text{ and } t\neq s \text{ we have } \frac{|x(t)- x(s)|}{\rho(|t-s|)} < \infty \right\}.
\end{equation} 
It is known that if we endow the above with the norm
\begin{equation}
	\|x\|_\rho := |x(0)| + \sup_{\stackrel{t,s\in [0,1]}{ t \neq s}} \frac{|x(t)- x(s)|}{\rho(|t-s|)},
\end{equation} 
then $C_\rho[0,1]$ becomes a Banach space. Of course, we have $C_\rho[0,1] \subseteq C[0,1]$ since for any $x\in C_\rho[0,1]$ we have
\begin{equation}
	|x(t) - x(t)| \leq \|x\|_\rho \; \rho(|t-s|).
\end{equation}
The continuity of $x$ follows from the fact that $\rho$ is continuous at zero. 

Consider the following integral equation of the Volterra type
\begin{equation}\label{eqn:MainEq}
	x(t) = \int_0^t G(t,s) f(s, x(s)) ds, \quad t \in[0,1].
\end{equation}
In what follows, we assume the following regularity of the kernel
\begin{equation}\label{eqn:AssumG}
	\tag{H1}
	G \in C([0,1]^2), \quad M_G := \max_{s,t\in[0,1]} |G(s,t)|, \quad G(\cdot, s) \in C_\rho[0,1], \quad \|G(\cdot, s)\|_\rho \leq K_G, \quad s \in [0,1],
\end{equation}
for some constants $M_G, K_G > 0$. That is, we assume a rather weak regularity of $G$ controlled by a given modulus of continuity $\rho$ uniformly with respect to the second variable. Note that the existence of $M_G$ is automatic, as it follows from the continuity of $G$ on $[0,1]^2$. 

The nonlinearity of the equation $f: [0,1]\times\mathbb{R} \mapsto \mathbb{R}$ is assumed to satisfy
\begin{equation}\label{eqn:AssumF}
	\tag{H2}
	|f(t,x) - f(t, y)| \leq \phi(|x-y|) \quad \text{for any} \quad t \in [0,1] \quad \text{and}\quad M_f := \sup\left\{|f(t,0)|: \, t\in[0,1]\right\} < \infty,
\end{equation}
where $\phi: \mathbb{R}_+ \mapsto \mathbb{R}_+$ is nondecreasing. 

\begin{rem}
	Generally, in various boundary value problems that appear in the literature, the kernel $G$ is actually the Green function corresponding to an underlying differential equation. It usually happens that $G(\cdot, s) \in H_\gamma[0,1]$ with $0<\gamma\leq 1$, that is, it is $\gamma$-H"older continuous. For that space, the corresponding modulus of continuity is $\rho(t) = t^\gamma$. However, there are $C_\rho[0,1]$ spaces that do not arise from the H\"older space.
\end{rem}

The main space in which we will look for solutions of \eqref{eqn:MainEq} is a certain generalization of the H\"older space. The functions belonging to that space have an integrable modulus of continuity. In \cite{appell2021holder,appell1987application} the authors considered the following definition.
\begin{df}
	Let $0<\alpha\leq 1$ and $\alpha \leq \beta < \infty$. For $0<s\leq 1$ put
	\begin{equation}
		j_{\alpha,\beta}(x, [0,s]) := \int_0^s \sigma^{-(\beta+1)} \omega(x, \sigma)^\frac{\beta}{\alpha} d\sigma.
	\end{equation}
	The space $J_{\alpha,\beta}[0,1]$ of continuous functions satisfying $j_{\alpha,\beta}(x, [0,1]) < \infty$ is called the \textbf{integral-type H\"older space}. In an analogous way, we defined the space $J_{\alpha,\beta}[a,b]$ for arbitrary $a<b$. 
\end{df}
\noindent It can also be proved that if we introduce the following norm
\begin{equation}
	\|x\|_{\alpha, \beta} = |x(0)| + j_{\alpha,\beta} (x, [0,1])^\frac{\alpha}{\beta},
\end{equation} 
we find that $(J_{\alpha,\beta}[0,1], \|\cdot\|_{\alpha,\beta})$ is a Banach space. 

\begin{rem}
	If we define the integral-type H\"older space with $\beta = \infty$ by
	\begin{equation}
		j_{\alpha,\infty}(x, [0,1]) := \sup\left\{\sigma^{-\alpha} \omega(x, \sigma): 0<\sigma\leq 1\right\},
	\end{equation}
	then, as can be shown \cite{appell1987application} (and also see Corollary \ref{cor:LInftyJBound} below), the corresponding space $J_{\alpha,\infty}$ is actually equal to $H_\alpha[0,1]$. Moreover, we have
	\begin{equation}
		\lim\limits_{\beta\rightarrow\infty} J_{\alpha,\beta}[0,1] = H_\alpha[0,1].
	\end{equation}
	That is, the H\"older space is recovered as a particular limit of $J_{\alpha,\beta}[0,1]$. This also justifies the name of the space considered.
\end{rem}

\begin{rem}
	The space $J_{\alpha,\beta}[0,1]$ with $0<\alpha\leq \beta\leq \infty$ is strictly larger than the H\"older space $H_\gamma[0,1]$ for any $0<\gamma<1$. 
\end{rem}

We will also work in the usual Lebesgue spaces $L^p[0,1]$ with $1\leq p \leq \infty$ with the norms denoted by $\|\cdot\|_p$. 

\subsection{Existence of solutions}
For the study of \eqref{eqn:MainEq} we define the following nonlinear operator
\begin{equation}\label{eqn:OpeartorT}
	(Tx)(t) := \int_0^t G(t,s) f(s, x(s)) ds, \quad t \in [0,1].
\end{equation}
The next result provides a bound on the modulus of continuity generated by $Tx$.
\begin{lem}\label{lem:ModulusT}
	Suppose that $x\in C[0,1]$, then for $0<\sigma \leq 1$ we have
	\begin{equation}\label{eqn:ModulusT}
		\omega(Tx, \sigma) \leq \left(\phi(\|x\|_\infty) + M_f\right) \left(K_G\rho(\sigma) + \sigma M_G\right)
	\end{equation}
\end{lem}
\begin{proof}
	We would like to estimate $\omega(Tx, \sigma)$ for $0<\sigma \leq 1$. Take $t,\tau \in [0,1]$ with $|t-\tau| \leq \sigma$. Without loss of generality, we assume that $t > \tau$. Then by \eqref{eqn:OpeartorT} and subtraction/addition of the integral of $G(\tau, s)f(s, x(s))$, we have
	\begin{equation}
		|(Tx)(t) - (Tx)(\tau)| \leq \int_0^\tau |G(t,s) - G(\tau, s)| |f(s, x(s))| ds + \int_\tau^t |G(t,s)| |f(s, x(s))| ds.
	\end{equation}
	Furthermore, by the regularity assumption on the kernel \eqref{eqn:AssumG} we further can deduce
	\begin{equation}
		\begin{split}
			|(Tx)(t) - (Tx)(\tau)| 
			&\leq K_G \int_0^\tau \rho(|t-s|) \left(|f(s, x(s)) - f(s, 0)| + |f(s,0)| \right)ds \\
			&+ M_G \int_\tau^t \left(|f(s, x(s)) - f(s, 0)| + |f(s,0)| \right) ds.
		\end{split}
	\end{equation}
	Since the modulus of continuity is nondecreasing and again using the regularity of $G$, we obtain
	\begin{equation}
		\begin{split}
			|(Tx)(t) - (Tx)(\tau)| 
			&\leq K_G \rho(\sigma) \int_0^\tau \left(|f(s, x(s)) - f(s, 0)| + |f(s,0)| \right)ds\\
			&+ M_G \int_\tau^t \left(|f(s, x(s)) - f(s, 0)| + |f(s,0)| ds \right).
		\end{split}
	\end{equation}
	Next, we use assumptions on $f$, that is, \eqref{eqn:AssumF}, to obtain
	\begin{equation}
		|(Tx)(t) - (Tx)(\tau)| \leq K_G \rho(\sigma) \int_0^\tau \left(\phi(|x(s)|) + M_f \right)ds + M_G \int_\tau^t \left(\phi(|x(s)|) + M_f \right)ds.
	\end{equation}
	Finally, since $\phi$ is nondecreasing, we have
	\begin{equation}
		|(Tx)(t) - (Tx)(\tau)| \leq \left(\phi(\|x\|_\infty) + M_f \right)\left(K_G\rho(\sigma) \tau + M_G |t-\tau|\right).
	\end{equation}
	But $0\leq \tau\leq 1$ and $|t-\tau|\leq \sigma$, therefore we have proved \eqref{eqn:ModulusT}. 
\end{proof}

We can see that $Tx$ is a uniformly continuous function with a modulus of continuity being the worse of $\rho(\sigma)$ (which is inherited from the kernel) and $\sigma$ (Lipschitz continuity). The next result states that $T$ maps $J_{\alpha,\beta}[0,1]$ to itself.
\begin{prop}\label{prop:TNorm}
	Suppose that $0<\alpha\leq 1$ and $\alpha \leq \beta < \infty$. Assume that
	\begin{equation}\label{eqn:AssumRho}
		\int_0^1 \sigma^{-(\beta+1)} \rho(\sigma)^\frac{\beta}{\alpha} d\sigma < \infty.
	\end{equation}
	Then, for any $x \in J_{\alpha,\beta}[0,1]$ we have $Tx \in J_{\alpha,\beta}[0,1]$ and
	\begin{equation}\label{eqn:TNorm}
		\|Tx\|_{\alpha,\beta} \leq \left(\phi(\|x\|_{\alpha,\beta}) + M_f\right)  \left(\int_0^1 \sigma^{-(\beta+1)} \left(K_G\rho(\sigma) + \sigma M_G \right)^\frac{\beta}{\alpha} d\sigma\right)^\frac{\alpha}{\beta}.
	\end{equation}
\end{prop}
\begin{proof}
	Let $x\in J_{\alpha,\beta}[0,1]$. Using Lemma \ref{lem:ModulusT} we estimate $j_{\alpha,\beta}(Tx, [0,1])$ as follows
	\begin{equation}\label{eqn:HolderT}
		j_{\alpha,\beta}(Tx, [0,1]) \leq \int_0^1 \sigma^{-(\beta+1)} \omega(Tx, \sigma)^\frac{\beta}{\alpha} d\sigma \leq \left(\phi(\|x\|_\infty) + M_f\right)^\frac{\beta}{\alpha} \int_0^1 \sigma^{-(\beta+1)} \left(K_G\rho(\sigma) + \sigma M_G\right)^\frac{\beta}{\alpha} d\sigma. 
	\end{equation}
	Therefore, by our assumption and the fact that $\|x\|_\infty \leq \|x\|_{\alpha,\beta}$, we have
	\begin{equation}
		j_{\alpha,\beta}(Tx, [0,1]) \leq \left(\max\{K_G, M_G\}\left(\phi(\|x\|_{\alpha,\beta}) + M_f\right)\right)^\frac{\beta}{\alpha} \int_0^1 \sigma^{-(\beta+1)} \left(\rho(\sigma) + \sigma\right)^\frac{\beta}{\alpha} d\sigma
	\end{equation}
	Now, due to the H\"older inequality $(a+b)^\gamma \leq 2^{\gamma-1} (a^\gamma + b^\gamma)$ for $\gamma \geq 1$, we have for $\beta/\alpha \geq 1$
	\begin{equation}
		\int_0^1 \sigma^{-(\beta+1)} \left(\rho(\sigma) + \sigma \right)^\frac{\beta}{\alpha} d\sigma \leq 2^{\frac{\beta}{\alpha}-1}\left(\int_0^1 \sigma^{-(\beta+1)} \rho(\sigma)^\frac{\beta}{\alpha} d\sigma + \int_0^1 \sigma^{-(\beta+1)} \sigma^\frac{\beta}{\alpha} d\sigma\right).
	\end{equation}
	The second integral is convergent since $\beta(1/\alpha-1) -1 > -1$, while the first is finite due to our assumption. Hence, $Tx \in J_{\alpha,\beta}[0,1]$ with \eqref{eqn:HolderT} being the optimal estimate. Since $Tx(0) = 0$, we prove \eqref{eqn:TNorm}. 
\end{proof}
Notice that due to \eqref{eqn:AssumG} the condition \eqref{eqn:AssumRho} means that $G(\cdot, s) \in J_{\alpha,\beta}[0,1]$ for each $s\in [0,1]$.
\begin{rem}
	The condition \eqref{eqn:AssumRho} is satisfied, for example, by $\rho(\sigma) = C\sigma^\gamma$ with $0<\gamma < \alpha$ as a simple integration shows. 
\end{rem}
\begin{rem}\label{rem:R0}
	Suppose that there exists $r_0 > 0$ such that
	\begin{equation}\label{eqn:AssumR0}
		\left(\phi(r_0) + M_f\right)  \left(\int_0^1 \sigma^{-(\beta+1)} \left(K_G\rho(\sigma) + \sigma M_G \right)^\frac{\beta}{\alpha} d\sigma\right)^\frac{\alpha}{\beta} < r_0.
	\end{equation}
	Then, $T(B_{r_0}^{\alpha,\beta}) \subseteq B_{r_0}^{\alpha,\beta}$, where $B_{r_0}^{\alpha,\beta}$ is the ball in $J_{\alpha,\beta}[0,1]$ centered at zero with radius $r_0 >0$. This follows from \eqref{eqn:TNorm} and the fact that $\phi$ is nondecreasing. 
\end{rem}

To show the existence of a solution to the equation \eqref{eqn:MainEq} in $J_{\alpha,\beta}[0,1]$, we will use the technique of using the measure of noncompactness.
\begin{df}
	The \textbf{Hausdorff measure of noncompactness} $\chi$ of a bounded set $X$ in a Banach space $B$ is defined by
	\begin{equation}
		\chi(M) := \inf\left\{\epsilon: \text{ X admits a finite } \epsilon-\text{net in }X\right\}.
	\end{equation}
\end{df}
Observe that, due to the definition, if $\chi(X) = 0$ then $X$ is relatively compact in the Banach space $B$. Moreover, as can be easily seen, in general it may be very difficult to obtain a explicit formula for the measure of noncompactness. In practice, one usually looks for bounds of $\chi$. The following is very useful in our setting. 
\begin{prop}[\cite{appell2021holder}, Proposition 6.3]
	For any bounded set $X$ in $J_{\alpha, \beta}$ define
	\begin{equation}
		c(X) = \limsup\limits_{s\rightarrow 0} \left(\sup_{x\in X} j_{\alpha,\beta}(x, [0,s])\right).
	\end{equation}
	Then, the Hausdorff measure of noncompactness $\chi(X)$ is related with $c(X)$ according to
	\begin{equation}
		2^{-\frac{\beta}{\alpha}} \chi(X) \leq c(x) \leq 2^\frac{\beta}{\alpha} \chi(X).
	\end{equation}
\end{prop}
Having these devices, we are ready to find the measure of noncompactness of a ball in $J_{\alpha,\beta}[0,1]$.
\begin{prop}\label{prop:Noncompactness}
	Assume \eqref{eqn:AssumRho} and suppose that there exists $r_0 > 0$ such that \eqref{eqn:AssumR0} holds. Then, for any $X\subseteq B_{r_0}^{\alpha,\beta}$ we have $c(T X) = 0$. 
\end{prop}
\begin{proof}
	Let $x\in B_{r_0}^{\alpha,\beta}$. By Lemma \ref{lem:ModulusT} and the same estimates as in Proposition \ref{prop:TNorm} we have 
	\begin{equation}
		\begin{split}
			j_{\alpha,\beta}&(Tx, [0,1]) 
			\leq \left(\max\{K_G , M_G\}\left(\phi(\|x\|_{\alpha,\beta}) + M_f\right)\right)^\frac{\beta}{\alpha} \int_0^s \sigma^{-(\beta+1)} \left(\rho(\sigma) + \sigma \right)^\frac{\beta}{\alpha} d\sigma \\
			&\leq 2^{\frac{\beta}{\alpha}-1}\left(\max\{K_G , M_G\}\left(\phi(\|x\|_{\alpha,\beta}) + M_f\right)\right)^\frac{\beta}{\alpha}\left(\int_0^s \sigma^{-(\beta+1)} \rho(\sigma)^\frac{\beta}{\alpha}d\sigma + \int_0^s\sigma^{-(\beta+1)}\sigma^\frac{\beta}{\alpha} d\sigma\right).
		\end{split}
	\end{equation}
	By assumption \eqref{eqn:AssumRho} the first integral is convergent, hence vanishing as $s\rightarrow 0^+$. On the other hand, the second one can be computed explicitly
	\begin{equation}
		\int_0^s\sigma^{-(\beta+1)}\sigma^\frac{\beta}{\alpha} d\sigma = \frac{s^{\beta\left(\frac{1}{\alpha}-1\right)}}{\beta\left(\frac{1}{\alpha}-1\right)} \rightarrow 0 \quad \text{as} \quad s\rightarrow 0^+.
	\end{equation}
	Therefore,
	\begin{equation}
		\begin{split}
			c(T B_{r_0}^{\alpha,\beta}) 
			&= \limsup\limits_{s\rightarrow 0}\left(\sup_{x\in B_{r_0}^{\alpha,\beta}} j_{\alpha,\beta}(Tx, [0,s])\right) \leq 2^{\frac{\beta}{\alpha}-1}\left(\max\{K_G, M_G\}\left(\phi(\|x\|_{\alpha,\beta}) + M_f\right)\right)^\frac{\beta}{\alpha} \\
			&\times\left(\int_0^s \sigma^{-(\beta+1)} \rho(\sigma)^\frac{\beta}{\alpha}d\sigma + \frac{s^{\beta\left(\frac{1}{\alpha}-1\right)}}{\beta\left(\frac{1}{\alpha}-1\right)}\right) \rightarrow 0 \quad \text{as} \quad s\rightarrow 0^+.
		\end{split}
	\end{equation}
	Therefore, $c(TX) = 0$ for any $X \subseteq B_{r_0}^{\alpha,\beta}$. The proof is complete. 
\end{proof}
We can now proceed to the main result regarding the existence.
\begin{thm}\label{thm:Existence}
	Assume \eqref{eqn:AssumRho} and suppose that there exists $r_0>0$ satisfying \eqref{eqn:AssumR0}. Then, the integral equation \eqref{eqn:MainEq} has a solution in $J_{\alpha,\beta}[0,1]$ with $0<\alpha<1$ and $\alpha\leq \beta<\infty$. 
\end{thm}
\begin{proof}
	The proofs relies on invoking Schauder's fixed point theorem for the operator $T$ defined in \eqref{eqn:OpeartorT}. By our assumption and Remark \ref{rem:R0} we know that $T$ maps a bounded set $B_{r_0}^{\alpha,\beta}$ into itself. Moreover, Proposition \ref{prop:Noncompactness} states that $\chi(T B_{r_0}^{\alpha,\beta}) = 0$, which implies that $T B_{r_0}^{\alpha,\beta}$ is relatively compact in $J_{\alpha,\beta}[0,1]$. We are left to show that $T: B_{r_0}^{\alpha,\beta} \mapsto B_{r_0}^{\alpha,\beta}$ is continuous. To this end, take a sequence $\{x_n\}_n \subseteq B_{r_0}^{\alpha,\beta}$ convergent to some $x \in \overline{B_{r_0}^{\alpha,\beta}}$ in the norm $\|\cdot\|_{\alpha,\beta}$. We will show that $Tx_n \rightarrow Tx$ in $J_{\alpha,\beta}[0,1]$. Let $t,\tau \in [0,1]$ with $\tau < t$ and $|t-\tau|\leq \sigma$. By the same estimates as in the proof of Lemma \ref{lem:ModulusT} we obtain
	\begin{equation}
		\begin{split}
			|(Tx_n)(t) - (Tx)(t) &- \left((Tx_n)(\tau) - (Tx)(\tau)\right)| \leq \int_0^\tau |G(t,s)-G(\tau,s)| |f(s,x_n(s)) - f(s, x(s))| ds\\
			&+ \int_\tau^t |G(t,s)| |f(s,x_n(s)) - f(s, x(s))| =: I_1 + I_2.
		\end{split}
	\end{equation}
	Taking into account the regularity of $G$ and $f$ as in \eqref{eqn:AssumG} and \eqref{eqn:AssumF} we obtain a bound for the first integral
	\begin{equation}
		I_1 \leq K_G \int_0^\tau \rho(|t-s|) \phi(\|x_n-x\|_\infty) ds \leq K_G \rho(\sigma) \phi(\|x_n-x\|_{\alpha,\beta}) \tau \leq K_G \rho(\sigma) \phi(\|x_n-x\|_{\alpha,\beta})
	\end{equation} 
	Similarly, the second integral can be estimated as follows 
	\begin{equation}
		\begin{split}
			I_2 
			&\leq \int_\tau^t |G(t,s)| \phi(\|x_n-x\|_{\alpha,\beta}) ds \leq K_G \, \phi(\|x_n-x\|_{\alpha,\beta}) (t-\tau) \\
			&\leq M_G \phi(\|x_n-x\|_{\alpha,\beta}) \sigma.
		\end{split}
	\end{equation}
	Therefore, by the definition of the modulus of continuity, we have
	\begin{equation}
		\omega(Tx_n - Tx, \sigma) \leq \phi(\|x_n-x\|_{\alpha,\beta}) \left(K_G\rho(\sigma) + M_G\sigma\right).
	\end{equation}
	Now, plugging the above into the definition of the $\|\cdot\|_{\alpha,\beta}$ norm we obtain
	\begin{equation}
		\|Tx_n - Tx\|_{\alpha,\beta} \leq \phi(\|x_n-x\|_{\alpha,\beta}) \left(\int_0^1 \sigma^{-(\beta+1)}\left(K_G\rho(\sigma) + M_G\sigma\right)^\frac{\beta}{\alpha}d\sigma \right)^\frac{\alpha}{\beta}.
	\end{equation}
	Using the H\"older inequality $(a+b)^\gamma \leq 2^{\gamma-1} (a^\gamma+b^\gamma)$ with $\gamma\geq 1$ as in the proofs of the above results, we see that our assumption \eqref{eqn:AssumRho} implies that
	\begin{equation}
		\|Tx_n - Tx\|_{\alpha,\beta} \leq C_{G,\rho,\alpha,\beta} \phi(\|x_n-x\|_{\alpha,\beta}),
	\end{equation}
	for some constant $C_{G,\rho,\alpha,\beta} > 0$. Therefore, passing to the limit with $n\rightarrow \infty$ and using the continuity of $\phi$ at zero, we conclude that $Tx_n \rightarrow Tx$ in $J_{\alpha,\beta}[0,1]$. Hence, $T$ is continuous. All hypotheses of Schauder fixed point theorem are then satisfied, therefore, there exists a fixed point of the operator $T$. The integral equation \eqref{eqn:MainEq} then has a solution.
\end{proof}

\paragraph{Example.} For the capillary rise equation \eqref{eqn:CapillaryRiseEq} we have
\begin{equation}
	G(t,s) = 1- e^{-(t-s)}, \quad f(s, x) = 1 - \sqrt{2x}.
\end{equation}
Therefore, since $G$ is Lipschitz while $f$ is $1/2$-H\"older continuous, we have
\begin{equation}
	M_G = 1, \quad K_G = 1, \quad M_f = 1, \quad \rho(t) = t, \quad \phi(t) = \sqrt{2t}. 
\end{equation}
Hence, after computing an elementary integral, the condition \eqref{eqn:AssumR0} becomes
\begin{equation}
	\frac{1}{\sqrt{2r_0}} + \frac{1}{r_0} <\frac{1}{2} \left(\beta\left(\frac{1}{\alpha}-1\right)\right)^\frac{\alpha}{\beta}. 
\end{equation}
Since the left-hand side decreases in $r_0$, for any fixed $0<\alpha\leq \beta$ there exists a unique $r_0 > 0$ such that the above holds (it can explicitly be found by solving a quadratic). \qed

\section{Collocation numerical methods}
In this section, our objective is to construct an efficient numerical method to solve \eqref{eqn:MainEq}. Due to its versatility, our choice is the colocation scheme, which we split into two families. The first is the normal \textit{collocation method} in which we approximate the solution of \eqref{eqn:MainEq} using a piecewise linear polynomial. This approach is very well suited to solve equations with minimal regularity such as those belonging to the space $J_{\alpha, \beta}[0,1]$ as in our case. The advantage is that we do not have to deal with the grid-wise pointwise values of the approximation, but rather we insist that the integral equation is satisfied at the previously chosen set of nodes. However, the difficulty is that we have to be able to provide a reliable estimate for the linear interpolation error in the respective function space. Note that in the case of low regularity, it is not desirable to consider higher-order than linear collocations. Since then, the error constant can become very large, leaving the order of convergence unchanged. On the other hand, when sufficient regularity is provided, it is very useful to implement a spectral collocation method, since it is well known that it can achieve exponential accuracy provided the solution is regular enough \cite{canuto2006spectral}. Having that in mind, in what follows, we present these two approaches to devise an efficient and fast numerical scheme to be used in various situations.

Introduce a grid of nodes that partition the interval $[0,1]$ 
\begin{equation}\label{eqn:Grid}
	0 = t_0 < t_1 < ... < t_N = 1,
\end{equation}
where $N \in \mathbb{N}$ is the total number of subintervals in $[0,1]$. For example, the important uniform grid is defined by 
\begin{equation}\label{eqn:UniGrid}
	t_n = n h, \quad h = \frac{1}{N}, \quad 0 \leq n \leq N.
\end{equation}
For the collocation scheme, we construct an approximation of $x$ requiring that it satisfies the equation and the initial condition at the nodes. Before we consider specific examples of the collocation schemes, we discuss linear interpolation in our integral H\"older space. 

\subsection{Linear interpolation in the $J_{\alpha,\beta}$ space}
In this section, we are working with the uniform grid \eqref{eqn:UniGrid}, but the generalization to a general discretization can be made very easily. Define the interpolation operator (which is a projection on the space of piecewise linear polynomials)
\begin{equation}\label{eqn:LinearInterpolation}
	I_h x(t) = \frac{1}{h} \left(x(t_{n+1}) (t- t_n) + x(t_n) (t_{n+1}- t)\right), \quad t \in [t_{n}, t_{n+1}], \quad 0\leq n \leq N-1, \quad x \in C[0,1].
\end{equation}
It is well known that for $C^2[0,1]$ smooth functions we have $\|I_h x - x\|_\infty \leq C h^2 \|x\|_{C^2[0,1]}$, where $C>0$ is an explicitly known function. Moreover, in \cite{caballero2025functional} we have shown that for functions of lower regularity, belonging to $H^{k,\gamma}[0,1]$, that is, having $k=0,1$-th derivative $\gamma$-H\"older continuous with $0<\gamma<1$, we have
\begin{equation}
	\|I_h x - x\|_\infty \leq C h^{k+\gamma} \|x\|_{H^{k,\gamma}[0,1]}.
\end{equation}
The proof of convergence of the collocation scheme is based on a proper error decomposition such that the interpolation error estimate can be used. Therefore, we have to find its supremum bound for solutions belonging to the space $J_{\alpha,\beta}$. The following lemma is crucial for this result.
\begin{lem}
	Let $x \in C[a,b]$ with $a<b$. Then, for any $\gamma, \delta >1$ we have
	\begin{equation}\label{eqn:FundamentalEstimate}
		\|x\|_\infty \leq C_{\gamma, \delta} (b-a)^\frac{\gamma-1}{\delta} \left(\int_{a}^{b} \sigma^{-\gamma} \omega(x, \sigma)^\delta d\sigma\right)^\frac{1}{\delta},
	\end{equation}
	where
	\begin{equation}
		C_{\gamma,\delta} = 
		\left(\frac{\gamma-1}{2^{\gamma-1}-1}\right)^\frac{1}{\delta} \left(\frac{1}{1-2^{-\frac{\gamma-1}{\delta-1}}}\right)^\frac{\delta-1}{\delta}.
	\end{equation}
\end{lem}
\begin{proof}
	Using a simple translation argument, we can assume that $x=x(t)$ is defined on $[0,T]$ with $T:=b-a$. If the integral in \eqref{eqn:FundamentalEstimate} is infinite, then the inequality is trivially satisfied. Assume otherwise and fix $t\in [0, T]$. We must relate the point value of $x(t)$ to its integrated value. To this end, we will use the dyadic decomposition technique used in harmonic analysis and constructive approximation theory. It is also closely related to Besov spaces (\cite{bahouri2011fourier}, Chapter 2 or \cite{devore1993constructive}, Chapter 2). Any point in $[0,T]$ can be uniquely expressed in the binary system as
	\begin{equation}
		t = T\sum_{k=0}^\infty \epsilon_k 2^{-k} =: \sum_{k=0}^\infty t_k \quad \text{where} \quad \epsilon_k \in \left\{0,1\right\}.
	\end{equation}
	Therefore, we can express the value of $x(t)$ as a telescoping series on exponentially decreasing intervals
	\begin{equation}
		x(t) = \sum_{k=1}^\infty (x(t_k) - x({t_{k-1}})).
	\end{equation}
	Hence, by the definition of the modulus of continuity, we have
	\begin{equation}\label{eqn:xEst}
		|x(t)| \leq \sum_{k=1}^\infty \omega(x, t_k - t_{k-1}) \leq \sum_{k=1}^\infty \omega(x, (b-a)2^{-k}).
	\end{equation}
	Now, since the modulus of continuity is nondecreasing, we immediately have for $\gamma \neq 1$
	\begin{equation}
		\begin{split}
			\int_{t2^{-k}}^{t2^{-k+1}} &\sigma^{-\gamma} \omega(x, \sigma)^\delta d\sigma \geq \omega(x, t2^{-k})^\delta \int_{t2^{-k}}^{t2^{-k+1}} \sigma^{-\gamma} d\sigma \\
			&= t^{1-\gamma}2^{(1-k)(1-\gamma)}\frac{2^{\gamma-1}-1}{\gamma-1} \omega(x, t2^{-k})^\delta,
		\end{split}
	\end{equation}
	or
	\begin{equation}
		\omega(x, t2^{-k}) \leq \left(t^{\gamma-1} 2^{(1-k)(\gamma-1)} \frac{\gamma-1}{2^{\gamma-1}-1} \int_{t2^{-k}}^{t2^{-k+1}} \sigma^{-\gamma} \omega(x, \sigma)^\delta d\sigma\right)^\frac{1}{\delta}.
	\end{equation}
	Therefore, from \eqref{eqn:xEst} we obtain
	\begin{equation}
		|x(t)| \leq t^\frac{\gamma-1}{\delta} \left(\frac{\gamma-1}{2^{\gamma-1}-1}\right)^\frac{1}{\delta} \sum_{k=1}^\infty 2^{(1-k)\frac{\gamma-1}{\delta}}\left(\int_{t2^{-k}}^{t2^{-k+1}} \sigma^{-\gamma} \omega(x, \sigma)^\delta d\sigma\right)^\frac{1}{\delta}.
	\end{equation}
	Or, using the H\"older inequality with $p = \delta/(\delta-1)$ and $q = \delta$, we further have
	\begin{equation}
		|x(t)| \leq t^\frac{\gamma-1}{\delta} \left(\frac{\gamma-1}{2^{\gamma-1}-1}\right)^\frac{1}{\delta} \left(\sum_{k=1}^\infty 2^{(1-k) \frac{\gamma-1}{\delta-1}}\right)^\frac{\delta-1}{\delta} \left(\sum_{k=1}^\infty \int_{t2^{-k}}^{t2^{-k+1}} \sigma^{-\gamma} \omega(x, \sigma)^\delta d\sigma\right)^\frac{1}{\delta}.
	\end{equation}
	By explicitly summing both series, we can simplify
	\begin{equation}
		|x(t)| \leq t^\frac{\gamma-1}{\delta} \left(\frac{\gamma-1}{2^{\gamma-1}-1}\right)^\frac{1}{\delta} \left(\frac{1}{1-2^{-\frac{\gamma-1}{\delta-1}}}\right)^\frac{\delta-1}{\delta} \left(\int_{0}^{t} \sigma^{-\gamma} \omega(x, \sigma)^\delta d\sigma\right)^\frac{1}{\delta}.
	\end{equation}
	Taking the supremum over $t\in [0,T]$ finishes the proof.
\end{proof}
Immediately, taking $\gamma = \beta + 1$ and $\delta = \beta/\alpha$ for $0<\alpha \leq \beta$ we can see that the following result is true. 
\begin{cor}\label{cor:LInftyJBound}
	Suppose that $x \in J_{\alpha,\beta}[0,T]$. Then,
	\begin{equation}\label{eqn:LInftyJBound}
		\|x\|_\infty \leq \left(\frac{\beta}{2^{\beta}-1}\right)^\frac{\alpha}{\beta} \left(1-2^{-\frac{\beta}{\frac{\beta}{\alpha}-1}}\right)^{-\frac{\beta - \alpha}{\beta}}T^\alpha \|x\|_{\alpha,\beta}.
	\end{equation}
	Moreover, for $\beta\rightarrow\infty$ we obtain
	\begin{equation}\label{eqn:LInftyJBoundInfty}
		\|x\|_\infty \leq \frac{T^\alpha}{2^{\alpha}-1} \sup_{0<\sigma\leq T}\frac{\omega(x,\sigma)}{\sigma^\alpha}.
	\end{equation}
\end{cor}
\begin{proof}
	We only have to show that in the limit $\beta\rightarrow\infty$ the bound becomes \eqref{eqn:LInftyJBoundInfty}. By simple rules of calculus, the prefactor is simply
	\begin{equation}
		\underbrace{\left(\frac{\beta}{2^{\beta}-1}\right)^\frac{\alpha}{\beta}}_{\rightarrow 2^{-\alpha}} \underbrace{\left(1-2^{-\frac{\beta}{\frac{\beta}{\alpha}-1}}\right)^{-\frac{\beta - \alpha}{\beta}}}_{\rightarrow (1-2^{-\alpha})^{-1}}.
	\end{equation}
	Now, write the norm as
	\begin{equation}
		\|x\|_{\alpha,\beta} = \left(\int_0^T \sigma^{-1} \left(\frac{\omega(x,\sigma)}{\sigma^\alpha}\right)^\frac{\beta}{\alpha}d\sigma\right)^\frac{\alpha}{\beta},
	\end{equation}
	and set $p:= \beta/\alpha$. Then, the above is the $p$-th norm of $\omega(x,\sigma)/\sigma^{\alpha}$ in a weighted Lebesgue space $L^p(\mu)$ with $d\mu := \sigma^{-1} d\sigma$. Because $\alpha$ is fixed, we have $p\rightarrow\infty$ when $\beta\rightarrow\infty$. Hence, due to the well-known results in passing from $L^p(\mu)$ to $L^\infty(\mu)$, we obtain
	\begin{equation}
		\|x\|_{\alpha,\beta} = \left\|\frac{\omega(x, \cdot)}{(\cdot)^\alpha}\right\|_{L^p(\mu)} \longrightarrow \left\|\frac{\omega(x, \cdot)}{(\cdot)^\alpha}\right\|_{L^\infty(\mu)} \quad \text{as} \quad p \longrightarrow \infty,
	\end{equation}
	which is equivalent to \eqref{eqn:LInftyJBoundInfty} after taking into account the limit of the prefactor. The proof is complete. 
\end{proof}

Now, we can formulate our result on the error of interpolation in the $J_{\alpha,\beta}[0,1]$ space. 
\begin{prop}\label{prop:LInftyJInterpolation}
	Let $x \in J_{\alpha,\beta}[0,1]$. Then, for its linear interpolation \eqref{eqn:LinearInterpolation} and $0<\alpha\leq \beta$ we have
	\begin{equation}
		\|x - I_h x\|_\infty \leq \left(\frac{\beta}{2^{\beta}-1}\right)^\frac{\alpha}{\beta} \left(1-2^{-\frac{\beta}{\frac{\beta}{\alpha}-1}}\right)^{-\frac{\beta - \alpha}{\beta}}h^\alpha \left(\int_0^h \sigma^{-\beta-1} \omega(x,\sigma)^\frac{\beta}{\alpha}\right)^\frac{\alpha}{\beta},
	\end{equation}
	while for $\beta\rightarrow\infty$ we obtain
	\begin{equation}
		\|x-I_h x\|_\infty \leq \frac{h^\alpha}{2^{\alpha}-1} \sup_{0<\sigma\leq h}\frac{\omega(x,\sigma)}{\sigma^\alpha}.
	\end{equation}
\end{prop}
\begin{proof}
	Take $t\in [t_n, t_{n+1}]$ for $0\leq n\leq N-1$ and notice that since $x(t) = (t-t_n)x(t) + (t_{n+1}-t)x(t)$ from \eqref{eqn:LinearInterpolation} we can write
	\begin{equation}
		\begin{split}
			|x(t) - I_h x(t)| 
			&= h^{-1}\left|(t-t_n)(x(t)-x(t_{n+1})) + (t_{n+1}-t) (x(t) - x(t_n))\right| \\
			&\leq 2\max\{|x(t)-x(t_{n+1})|, |x(t) - x(t_n)|\}.
		\end{split}
	\end{equation}
	But the modulus of continuity is not changed when we subtract a constant from the function, hence we can use Corollary \ref{cor:LInftyJBound} to obtain
	\begin{equation}
		|x(t) - I_h x(t)| \leq \left(\frac{\beta}{2^{\beta}-1}\right)^\frac{\alpha}{\beta} \left(1-2^{-\frac{\beta}{\frac{\beta}{\alpha}-1}}\right)^{-\frac{\beta - \alpha}{\beta}}h^\alpha \left(\int_0^h \sigma^{-\beta-1} \omega(x,\sigma)^\frac{\beta}{\alpha}\right)^\frac{\alpha}{\beta}.
	\end{equation}
	Since this is true for any subinterval $[t_n, t_{n+1}]$ we conclude that it also holds for the supremum taken on the left-hand side. The proof is complete. 
\end{proof}

\paragraph{Example.} Take a $\gamma$-H\"older function with $0<\alpha<\gamma<1$. Then $\omega(x,\sigma) = L \sigma^\gamma$. By simple computations we have
\begin{equation}
	\begin{split}
		\|x - I_h x\|_\infty 
		&\leq L\left(\frac{\beta}{2^{\beta}-1}\right)^\frac{\alpha}{\beta} \left(1-2^{-\frac{\beta}{\frac{\beta}{\alpha}-1}}\right)^{-\frac{\beta - \alpha}{\beta}}h^\alpha \left(\frac{h^{\beta(\gamma/\alpha-1)}}{\beta(\frac{\gamma}{\alpha}-1)}\right)^\frac{\alpha}{\beta}\\
		&\leq L\left(2^{\beta}-1\right)^{-\frac{\alpha}{\beta}} \left(1-2^{-\frac{\beta}{\frac{\beta}{\alpha}-1}}\right)^{-\frac{\beta - \alpha}{\beta}}\left(\frac{\gamma}{\alpha}-1\right)^{-\frac{\alpha}{\beta}}h^{\gamma}.
	\end{split}
\end{equation}
Which is the correct order for the error of the linear interpolation for H\"older functions (see \cite{okrasinska2025functional}). Note also that for $\beta\rightarrow\infty$ we can even take $\gamma = \alpha$, which further corresponds to the usual results for linear interpolation. That is to say, having limited smoothness, we cannot get past the order of regularity of the interpolated function. 

\subsection{Linear collocation scheme}
Now, we are ready to set up the linear collocation scheme. According to the definition of the method, we require that the integral equation \eqref{eqn:MainEq} be satisfied at the nodes of our discretization \eqref{eqn:Grid}. By $x_h$ define the \emph{continuous piecewise linear} approximation to $x$. The collocation conditions are, thus, the following
\begin{equation}\label{eqn:LinearCollocationScheme}
	\begin{cases}
		x_h \text{ is a linear function on each } [t_n, t_{n+1}], & 0 \leq n < N, \\
		x_h(t_n^+) = x_h(t_{n+1}^-), & 0 < n < N, \\
		x_h(t_n) = (Tx_h)(t_n), & 0 \leq n \leq N, \\
	\end{cases}
\end{equation}
where the operator $T$ is defined in \eqref{eqn:OpeartorT}. The above defines $(N-1) + (N + 1) = 2N$ conditions. Since we have $N$ subintervals and on each the approximation is linear, we have $N \times 2 = 2N$ coefficients to find. Therefore, the number of unknowns is the same as the number of equations. To devise a convenient form for the implementation, it is natural to introduce the \textbf{Lagrange basis} (tent functions)
\begin{equation}
	\begin{split}
		L_n(t) &:= 
		\begin{cases} 
			1-\dfrac{|t-t_n|}{h}, & t \in [t_{n-1}, t_{n+1}], \\
			0, & \text{elsewhere},
		\end{cases}
		\quad 0 < n < N, \\
		L_0(t) &:= \begin{cases} 
			\dfrac{t_1-t}{h}, & t \in [t_0, t_1], \\
			0, & \text{elsewhere},
		\end{cases}
		\quad
		\quad L_N(t) := \begin{cases} 
			\dfrac{t-t_{N-1}}{h}, & t \in [t_{N-1}, t_{N}], \\
			0, & \text{elsewhere}.
		\end{cases}
	\end{split}
\end{equation}
Of course, $L_n(x_m) = \delta_{nm}$ and thus the linear interpolation can conveniently be written as
\begin{equation}
	I_h x(t) = \sum_{i=0}^{N} x(t_n) L_n(t).
\end{equation}
In particular, expansion of a constant function $x(t) \equiv 1$ gives
\begin{equation}\label{eqn:LinearInterpolationConstant}
	\left|\sum_{n=0}^{N} L_n(t)\right| = 1.
\end{equation}
Moreover, any piecewise linear function (not necessarily the interpolant) can be decomposed into a combination of Lagrange functions. In particular, the collocation approximation $x_h$. We thus have,
\begin{equation}\label{eqn:LagrangeExpansion}
	x_h(t) = \sum_{n=0}^{N} x_h(t_n) L_n(t), \quad t \in [0,1].
\end{equation} 
By this and \eqref{eqn:MainEq} we have $x_h(t_0) = 0$ and 
\begin{equation}
	\begin{split}
		x_h(t_n) &= \int_0^{t_n} G(t_n, s) f(s, x_h(s)) ds = \sum_{i = 0}^{n-1} \int_{t_i}^{t_{i+1}} G(t_n, s) f(s, x_h(s)) ds \\
		&= \sum_{i = 0}^{n-1} \int_{t_i}^{t_{i+1}} G(t_n, s) f(s, x(t_i) L_i(s) + x(t_{i+1}) L_{i+1}(s)) ds, \quad 1 \leq n \leq N. 
	\end{split}
\end{equation}
Therefore, the collocation approximation can be obtained by solving the following nonlinear system
\begin{equation}\label{eqn:LinearCollocationSystem}
	\textbf{x}_h = \textbf{F}(\textbf{x}_h),
\end{equation}
where $\textbf{x}_h = (x_h(t_n))_{n=1}^N$ and $\textbf{F}(\textbf{x}_h) = \left(F_n(x_h(t_1), ..., x_h(t_n))\right)_{n=1}^N$, where
\begin{equation}\label{eqn:LinearCollocationSystemCoeff}
	F_n(x_h(t_1), ..., x_h(t_n)) = \sum_{i = 0}^{n-1} \int_{t_i}^{t_{i+1}} G(t_n, s) f(s, x(t_i) L_i(s) + x(t_{i+1}) L_{i+1}(s)) ds. 
\end{equation} 
Note that, due to our formulation, the continuity conditions are automatically satisfied. Along with the known initial condition $x_h(0) = 0$ we reduce the number of equations from $2N$ to $N$. Of course, in each step, we have to discretize the integral appearing in \eqref{eqn:LinearCollocationSystemCoeff} with, for example, trapezoid quadrature since it is exact for piecewise polynomial integrands (and having second order for others). As we can see, to compute a value $x_h(t_n)$ we have to know all the previous values of the approximation. Therefore, the system \eqref{eqn:LinearCollocationSystem} can be solved iteratively starting with specifying the starting value $x_h(0) = 0$. 

Having defined the linear collocation scheme, we can proceed to prove that it actually converges to the exact solution of \eqref{eqn:MainEq}. 
\begin{thm}[Convergence of the linear collocation scheme]\label{thm:LinearCollocationConvergence}
	Let the assumptions of Theorem \ref{thm:Existence} be satisfied. Suppose that $\phi$ from \eqref{eqn:AssumF} satisfies $\phi(x) \leq L x$ with $L>0$. Then, we have
	\begin{equation}\label{eqn:LinearCollocationError}
		\|x - x_h\|_\infty \leq C_{G,f} \left(\frac{\beta}{2^{\beta}-1}\right)^\frac{\alpha}{\beta} \left(1-2^{-\frac{\beta}{\frac{\beta}{\alpha}-1}}\right)^{-\frac{\beta - \alpha}{\beta}}h^\alpha \left(\int_0^h \sigma^{-\beta-1} \omega(x,\sigma)^\frac{\beta}{\alpha}\right)^\frac{\alpha}{\beta},
	\end{equation}
	for some constant $C_{G,f}>0$ dependent only on $G$ and $f$. 
\end{thm}
\begin{proof}
	By the definition of the collocation scheme \eqref{eqn:LinearCollocationScheme} and the interpolation operator \eqref{eqn:LinearInterpolation}, since $x_h$ is a piecewise linear polynomial, we have for any $t\in[0,1]$
	\begin{equation}
		x_h(t) = I_h x_h(t) = \sum_{n=1}^N L_n(t) x_h(t_n) = \sum_{n=1}^N L_n(t) Tx_h(t_n) = I_h Tx_h (t).
	\end{equation}	
	Let $e_h := x - x_h$ be the error of the scheme. Then, by the above, we have
	\begin{equation}
		e_h(t) = x(t) - x_h(t) = Tx(t) - I_h Tx_h(t) = Te_h(t) + Tx_h(t) - I_h T x_h(t).
	\end{equation}
	Therefore, by the definition of the operator $T$ as in \eqref{eqn:OpeartorT} and the regularity assumptions \eqref{eqn:AssumG}-\eqref{eqn:AssumG}, we have
	\begin{equation}
		|e_h(t)| \leq K_G \int_0^t \phi(|e_h(s)|) ds + \rho_h(t),
	\end{equation}
	where 
	\begin{equation}
		\rho_h(t) := |Tx_h(t) - I_h T x_h(t)|. 
	\end{equation}
	But, due to the assumption $\phi(|e_h(s)|) \leq L |e_h(s)|$, and hence
	\begin{equation}
		|e_h(t)| \leq K_G L \int_0^t |e_h(s)|ds + \rho_h(t).
	\end{equation} 
	Now, by the classical Gronwall inequality, for some new constant $D_{G,f}$ we obtain
	\begin{equation}
		|e_h(t)| \leq \rho_h(t) + D_{G,f} \int_0^t \rho_h(s) ds \leq (1 + D_{G,f}) \|\rho_h\|_\infty. 
	\end{equation}
	Hence, setting $C_{G,f} := 1+ D_{G,f}$,
	\begin{equation}
		\|e_h\|_\infty \leq C_{G,f} \|\rho_h\|_\infty. 
	\end{equation}	
	But due to Proposition \ref{prop:TNorm} we have $Tx \in J_{\alpha,\beta}[0,1]$, and therefore from Proposition \ref{prop:LInftyJInterpolation} we can obtain the estimate on the interpolation error $\|\rho_h\|_\infty$. The proof concludes. 
\end{proof}
The linear collocation scheme thus converges with an error of the same form as the interpolation error. This corresponds to similar results of convergence in more regular and common spaces (see \cite{okrasinska2025functional}).

\subsection{Spectral collocation scheme}
The linear collocation scheme devised in the previous subsection is a versatile method for solving \eqref{eqn:MainEq} in spaces without much regularity. However, it becomes inefficient when the solution is actually very smooth. On the other end of the spectrum of numerical methods we have spectral collocation schemes that are well-suited for such a case. Here, we describe a version of the spectral method based on Lagrange orthogonal polynomials.

Spectral collocation methods use expansions in the basis of orthogonal polynomials. It is well known that, provided sufficient smoothness, their accuracy can be exponential (so-called \textit{spectral accuracy}). Instead of using a uniform grid, we retain the freedom to approximate the unknown solution $x$ of \eqref{eqn:MainEq} by an approximation $x^N$ that we collocate with the equation on an arbitrarily spaced grid \eqref{eqn:Grid}. In contrast to the linear collocation scheme described above, we now would like to collocate a single (and not piecewise) polynomial of degree $N$. To avoid instability, we have to choose a non-uniform grid to be chosen below. That is, the \textbf{spectral collocation approximation} is defined by
\begin{equation}\label{eqn:SpectralCollocationScheme}
	x^N(t_n) = (Tx^N)(t_n), \quad x^N(0) = 0, \quad 0< n \leq N, \quad x^N \in \mathbb{P}_N[0,1].
\end{equation}
To facilitate convergence analysis, we would like to cast this scheme into a unified framework of functional analysis. More specifically, we want to derive a weak form of the above collocation equations. 

Since our initial condition vanishes, we will find a spectral approximation of the solution in the following space
\begin{equation}
	V_N := \left\{v \in \mathbb{P}_N[0,1]: v(0) = 0 \right\} \subseteq L^2[0,1].
\end{equation}
Any polynomial from the above space can be expanded in the Lagrange basis, analogously as above. However, now instead of using $N$-linear polynomials, we use one of $N$th degree with collocation points $\left\{t_n\right\}_{n=0}^N$ (not necessarily uniform). That is, the \textbf{$N$-th degree Lagrange basis} is defined by 
\begin{equation}
	L^N_n(t) = \prod_{\substack{0\leq i \leq N \\i \neq n}} \frac{t - t_i}{t_n - t_i}.
\end{equation} 
As can also be checked, $L^N_n(t_m) = \delta_{nm}$. Hence, we can define the interpolation operator $I^N: C[0,1] \mapsto \mathbb{P}_N[0,1]$
\begin{equation}
	I^N y(t) = \sum_{n=1}^N y(t_n) L^N_n(t), \quad y \in C[0,1].
\end{equation}
Now, define the \textbf{continuous inner product} on by
\begin{equation}
	(x,y) := \int_0^1 x(t) y(t) dt, \quad x, y \in C[0,1],
\end{equation}
and its \textbf{discrete version} by
\begin{equation}\label{eqn:DiscreteInnerProduct}
	(x,y)_N := \sum_{n=0}^N x(t_n) y(t_n) w_n.
\end{equation}
We immediately see that $(L^N_i, L^N_j)_N = \delta_{i j} w_i$ so that the Lagrange basis is orthogonal with respect to this inner product.  

Furthermore, we would like to choose the nodes $t_n$ and the weights $w_n$ such that the discrete inner product would agree with the continuous one when acting on a space of polynomials of maximal degree. This ensures spectral accuracy. As is well-known from numerical analysis, these are Gauss-Legendre nodes and weights scaled to the interval $[0,1]$ (see \cite{canuto2006spectral})
\begin{equation}\label{eqn:SpectralCollocationLegendreNodes}
	\left\{t_n\right\}_{n=0}^N := \text{zeros of }P_N(2t-1), \quad \left\{w_n\right\}_{n=0}^N := \left\{\frac{2}{(1-t_n^2) P'_N(2t_n-1)}\right\}_{n=0}^N.
\end{equation}
For which we have
\begin{equation}\label{eqn:DiscreteContinuousInner}
	(x,y)_N = (x,y), \quad xy \in \mathbb{P}_{2N+1}.
\end{equation}
Here, Legendre polynomials can be obtained via standard recurrence
\begin{equation}
	(n+1)P_{n+1}(t) = (2n+1)P_n - n P_n(t), \quad P_0(x) = 1, \quad P_1(x) = x.
\end{equation}
Our scheme \eqref{eqn:SpectralCollocationScheme} can now be cast in the weak form by noting that $\left\{L^N_n\right\}_{n=0}^N$ spans $V_N$. The integral operator $T$ in \eqref{eqn:SpectralCollocationScheme} can be discretized by the Gaussian quadrature \emph{on each} interval $[0,t_n]$ by choosing appropriate nodes $\{t_i^{(n)}\}_{i=0}^N$ and weights $\{w_i^{(n)}\}_{i=0}^N$ (they are rescaled values of \eqref{eqn:SpectralCollocationLegendreNodes})
\begin{equation}\label{eqn:OperatorTN}
	T^N x(t_n) = \sum_{i=0}^N G(t, t_i^{(n)}) f(t_i^{(n)}, x(t_i^{(n)})) w_i^{(n)}.
\end{equation}
Therefore, our fully discretized scheme reads
\begin{equation}\label{eqn:SpectralCollocationSchemeWeak}
	(x^N, v)_N = (T^Nx^N, v)_N, \quad v \in V_N.
\end{equation}
The goal of the spectral collocation scheme is to find $x^N(t_n)$ with \eqref{eqn:SpectralCollocationLegendreNodes} such that the above is satisfied. Also, due to the choice of nonuniform nodes we obtain optimal bounds for the interpolation (see \cite{canuto2006spectral})
\begin{equation}\label{eqn:SpectralInterpolationError}
	\|x - I^N x\|_2 \leq C N^{-m} \|x^{(m)}\|_2, \quad \|x - I^N x\|_\infty \leq C N^{\frac{1}{2}-m} \|x^{(m)}\|_\infty, \quad x \in C^m[0,1],
\end{equation}
for which the spectral accuracy is evident. Moreover, the error of the Gaussian quadrature is of exactly the same form
\begin{equation}\label{eqn:GaussianQuadratureError}
	\|Tx - T^N x\|_2 \leq C N^{-m} \|x^{(m)}\|_2, \quad \|Tx - T^N x\|_\infty \leq C N^{\frac{1}{2}-m} \|x^{(m)}\|_\infty.
\end{equation}
For the implementation, we observe that on our nodes \eqref{eqn:SpectralCollocationLegendreNodes} we use the discrete operator $T^N$ defined in \eqref{eqn:OperatorTN} and interpolate when the sub-node does not belong to $\{t_n\}_{n=0}^N$, hence 
\begin{equation}
	x^N(t_n) = \sum_{i=1}^N G(t_n, t_i^{(n)}) f(t_i^{(n)}, x^N(t_i)) w_i^{(n)}, \quad 0< n \leq N.
\end{equation}
The above gives a system of $N\times N$ algebraic equations for $\{x_n\}_{n=0}^N$ that can be solved numerically using Newton or fixed-point iterations. 

Having built the framework for the spectral collocation scheme, we can proceed to the convergence proof.
\begin{thm}\label{thm:SpectralCollocationScheme}
	Let $x \in C^m[0,1]$ be a solution of \eqref{eqn:MainEq} with $m > 0$ and $x^N$ its spectral collocation approximation defined in \eqref{eqn:SpectralCollocationSchemeWeak}. Then,
	\begin{equation}
		\|x-x^N\|_2 \leq C_{G,f} N^{-m} \|x^{(m)}\|_2, \quad \|x-x^N\|_\infty \leq C_{G,f} N^{\frac{1}{2}-m} \|x^{(m)}\|_\infty, 
	\end{equation}
	for some constant $C_{G,f}> 0$ dependent only on $G$ and $f$. 
\end{thm}
\begin{proof}
	Essentially in the same way as in the proof of Theorem \ref{thm:LinearCollocationConvergence} we can obtain the pointwise bound for the error $e^N := x - x^N$
	\begin{equation}
		|e^N(t)| \leq K_G L \int_0^t |e^N(s)| ds + \rho^N(t),
	\end{equation}
	where 
	\begin{equation}
		\rho^N(t) := |I^N Tx^N(t) - Tx^N(t)| + |T x^N(t)-T^N x^N(t)|.
	\end{equation}
	Therefore, by the same Gronwall inequality, we obtain
	\begin{equation}
		|e^N(t)| \leq \rho^N(t) + D_{G,f} \int_0^t \rho^N(s) ds. 
	\end{equation}
	For the $L^\infty$ estimate we just maximize the above to obtain $\|e^N\|_\infty \leq C_{G,f} \|\rho^N\|_\infty$, while for the $L^2$ norm we use Cauchy-Schwarz inequality
	\begin{equation}
		|e^N(t)| \leq \rho^N(t) + D_{G,f} \|\rho^N\|_2.  
	\end{equation}
	Therefore, there is a constant such that
	\begin{equation}
		\|e^N\|_2 \leq C_{G,f} \|\rho^N\|_2. 
	\end{equation}
	Since, by the assumption, $Tx^N$ is sufficiently smooth, we can use \eqref{eqn:SpectralInterpolationError} and \eqref{eqn:GaussianQuadratureError} and conclude the proof.  
\end{proof}
We have thus devised two spectral collocation schemes to solve \eqref{eqn:MainEq}: one designed for solution with limited smoothness and the other for smooth ones. In the following, we practically verify our findings.  

\section{Numerical illustration}
Here, we verify our theoretical findings with numerical experiments. We consider two classes of solutions to \eqref{eqn:MainEq}: non-smooth and smooth, to highlight the features of our two collocation schemes. In each of these cases, we consider explicit solutions to \eqref{eqn:MainEq} and based on these we compute the order of convergence. 

Note that when we put $f(s, x(s)) = \widetilde{g}(s) + \widetilde{f}(s, x(s))$ we can transform our main equation \eqref{eqn:MainEq} into
\begin{equation}
	x(t) = \int_0^t G(t, s) \widetilde{f}(s, x(s)) ds + g(t), \quad g(t) := \int_0^t G(t, s) \widetilde{g}(s) ds.
\end{equation}
Therefore, without loss of generality, we can consider the above ''nonhonogeneous'' equation. This form allows us to construct exact solutions more easily and hence test our numerical methods without the need of estimating orders of convergence via extrapolation or similar techniques. 

All the computations described in this section had been conducted with a multithreaded implementation of the schemes on a laptop computer with 12th Gen Intel(R) Core(TM) i7-1265U CPU at 1.80 GHz. The code was implemented in \texttt{Julia 1.11} programming language \cite{bezanson2017julia}. 

\subsection{Non-smooth solutions}
We choose two examples to test our linear collocation scheme \eqref{eqn:LinearCollocationScheme}. The first one, is purely H\"older continuous produced by an equation with strong H\"older nonlinerity similar as in the capillary rise equation \eqref{eqn:CapillaryRiseEq}
\begin{equation}\label{eqn:NumericalHolder}
	\begin{split}
		&x(t) = \sqrt{\frac{1}{2}-\left|t-\frac{1}{2}\right|}, \quad G(t,s) = 1, \quad \widetilde{f}(s, x) = 1 + \sqrt{x}, \\
		&g(t) =
		\begin{cases}
			\displaystyle
			\sqrt{\tfrac12 - t}
			\;-\;\Bigl(t \;+\;\tfrac{4}{5}\,t^{5/4}\Bigr),
			&0 \le t \le \tfrac12,
			\\[2ex]
			\displaystyle
			\sqrt{t - \tfrac12}
			\;-\;\Bigl(t \;+\;\tfrac{8}{5}\Bigl(\tfrac12\Bigr)^{5/4}
			\;-\;\tfrac{4}{5}\,(1 - t)^{5/4}\Bigr),
			&\tfrac12 \le t \le 1.
		\end{cases}
	\end{split}
\end{equation}
Notice the pronounced cusp at $t=1/2$ making the solution nonsmooth (see Fig. \ref{fig:NonsmoothSolutions}). In the second example, we construct a logarithmic solution that belongs to $J_{\alpha,\beta}[0,1]$ but is not H\"older
\begin{equation}\label{eqn:NumericalLogarithm}
	x(t) = \left(\ln\frac{e}{t}\right)^{-1}, \quad G(t,s) = t-s, \quad \widetilde{f}(s,x) = 1+e^{-x}, \quad g(t) = x(t) - Tx(t),
\end{equation}
where the integration needed to obtain $g(t)$ is conducted numerically with Gaussian quadrature. The solution is not smooth and has a very steep gradient at $t = 0$ (see Fig. \ref{fig:NonsmoothSolutions} for the plot).  

\begin{figure}
	\centering
	\includegraphics[scale = 0.7]{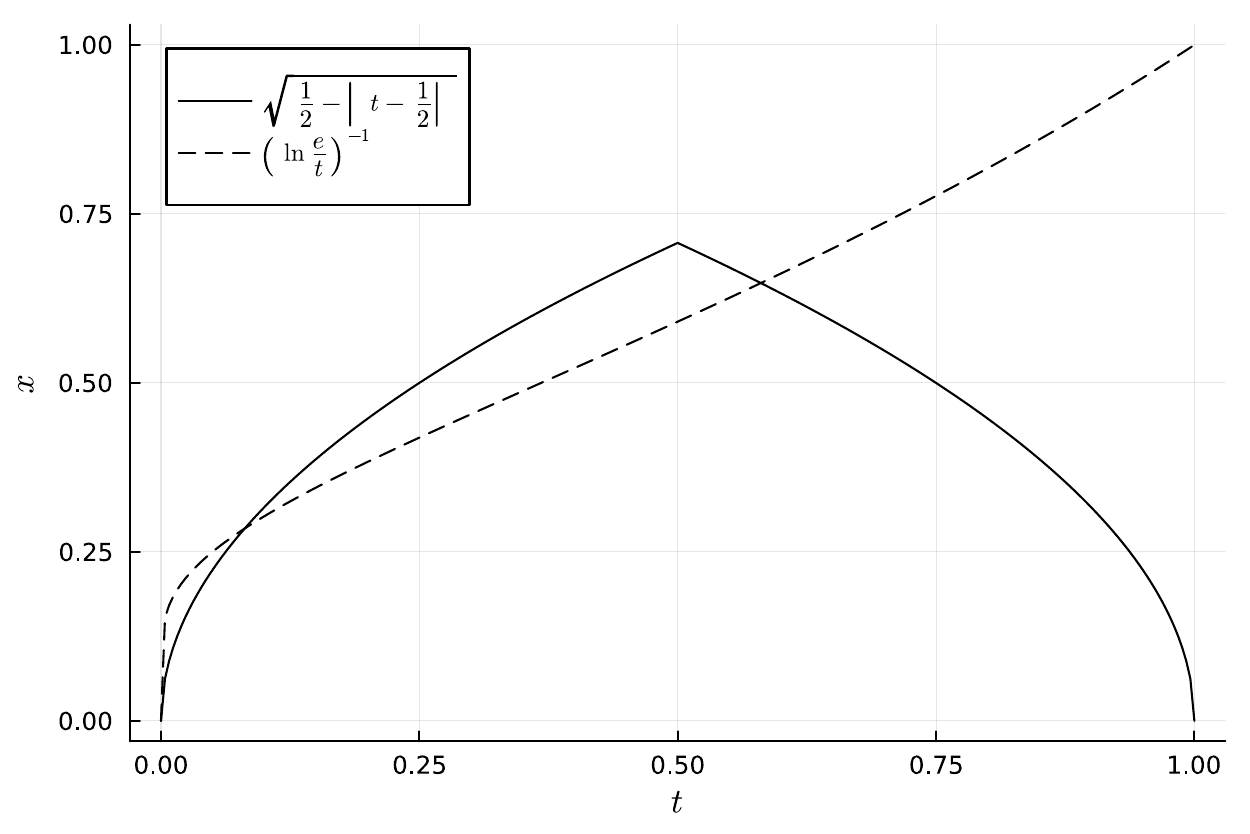}
	\caption{Exact non-smooth solutions chosen for testing the piecewise linear collocation scheme: \eqref{eqn:NumericalHolder} (solid line), \eqref{eqn:NumericalLogarithm} (dashed line).}
	\label{fig:NonsmoothSolutions}
\end{figure}

As can be seen in Fig. \ref{fig:ErrorNonsmooth} we have been able to obtain an order of convergence slightly larger than $1$ in both of our examples. Moreover, we have tested our scheme with several other examples, here not specified, and noted exactly the same behavior. According to Theorem \ref{thm:LinearCollocationConvergence}, we know that the scheme is convergence with an order at least bounded by \eqref{eqn:LinearCollocationError}. As is evident, in practice the error can be even higher than that which shows the robustness of our method.

\begin{figure}
	\centering
	\includegraphics[scale = 0.42]{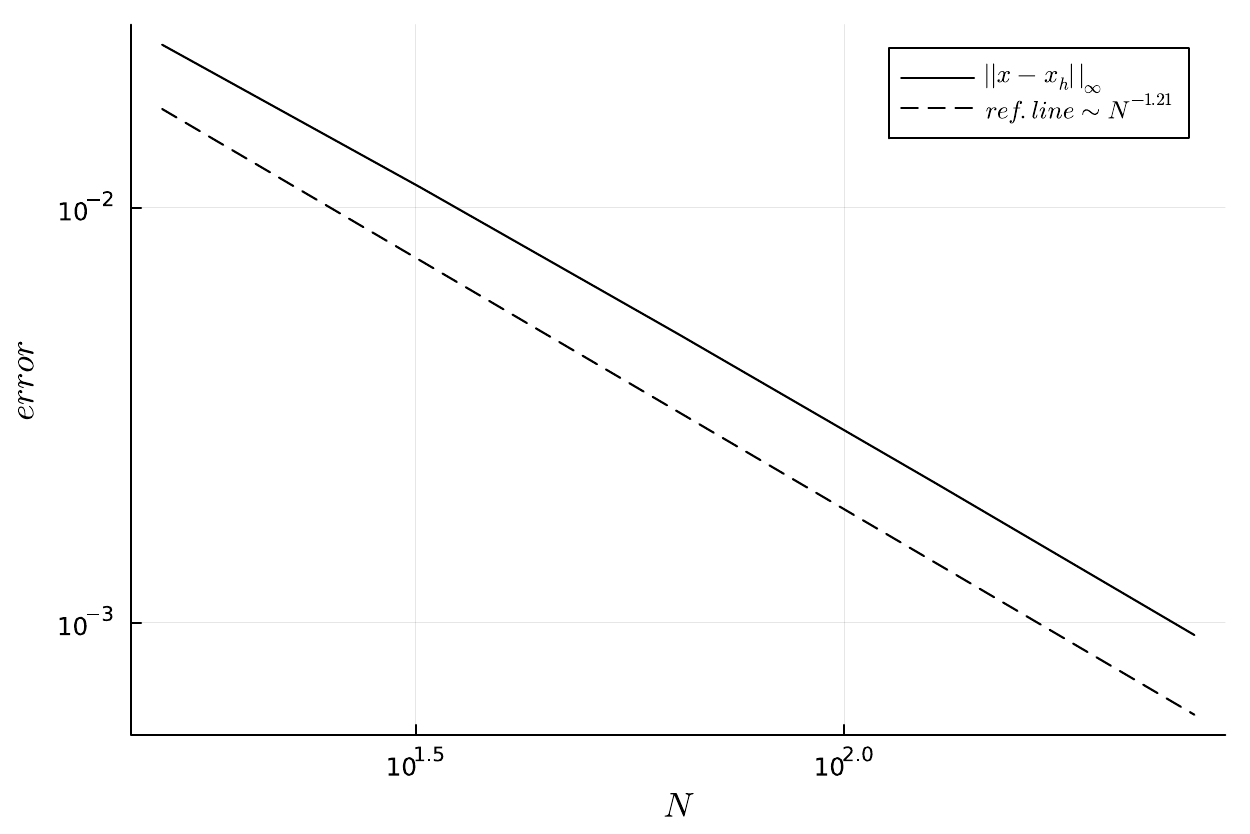}
	\includegraphics[scale = 0.42]{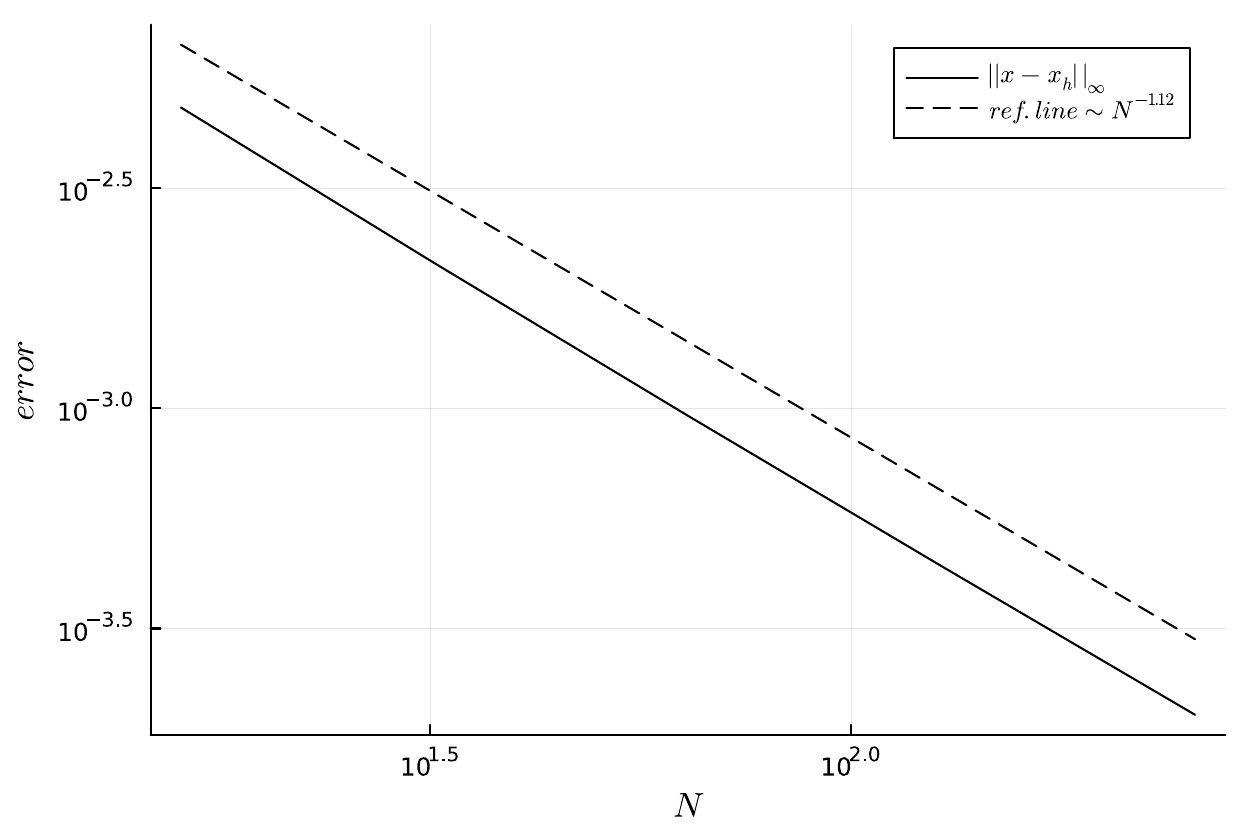}
	\caption{$L^\infty$ error of the linear collocation scheme applied to \eqref{eqn:NumericalLogarithm}. On the left: H\"older solution \eqref{eqn:NumericalHolder}, on the right: logarithmic solution \eqref{eqn:NumericalLogarithm}.}
	\label{fig:ErrorNonsmooth}
\end{figure}

\subsection{Smooth solutions}
Here, we test our spectral collocation scheme \eqref{eqn:SpectralCollocationScheme} on the logarithmic example \eqref{eqn:NumericalLogarithm} to assess how it performs in the non-smooth scenario, and the following very smooth solution
\begin{equation}\label{eqn:NumericalSmooth}
	x(t) = \int_0^t e^{t-s} (1 + x(s)) ds, \quad x(s) = \frac{1}{2}\left(e^{2t}-1\right), \quad 0 \leq t \leq 1. 
\end{equation}

The error of the spectral collocation scheme applied to the smooth problem \eqref{eqn:NumericalSmooth} is depicted in Fig. \ref{fig:ErrorSmooth}. Notice that the plot is on the semilog scale, on which the error decreases linearly with the number of degrees of freedom $N$. This indicates spectral (exponential) accuracy. At $N=14$, the error saturates at the machine precision. Therefore, we can conclude that only several degrees of freedom are necessary to obtain an extremely good approximation to the solution. We have tested this scheme along with some other examples with essentially the same result. In each of these experiments, the average computation time \emph{did not exceed} $1-2$ seconds!

On the other hand, the spectral method does not perform well for the non-smooth problem \eqref{eqn:NumericalLogarithm}. This is, of course, anticipated, since according to Theorem \ref{thm:SpectralCollocationScheme} the error is bounded by a sufficiently high derivative. Since in this example, all derivatives blow up, we cannot expect that the scheme will converge. As shown in Fig. \ref{fig:ErrorSmooth}, the error saturates to some positive value. This shows that using piecewise linear solvers for non-smooth solutions is the method of choice, while the spectral scheme is optimal for smooth solutions. 

\begin{figure}
	\centering
	\includegraphics[scale = 0.42]{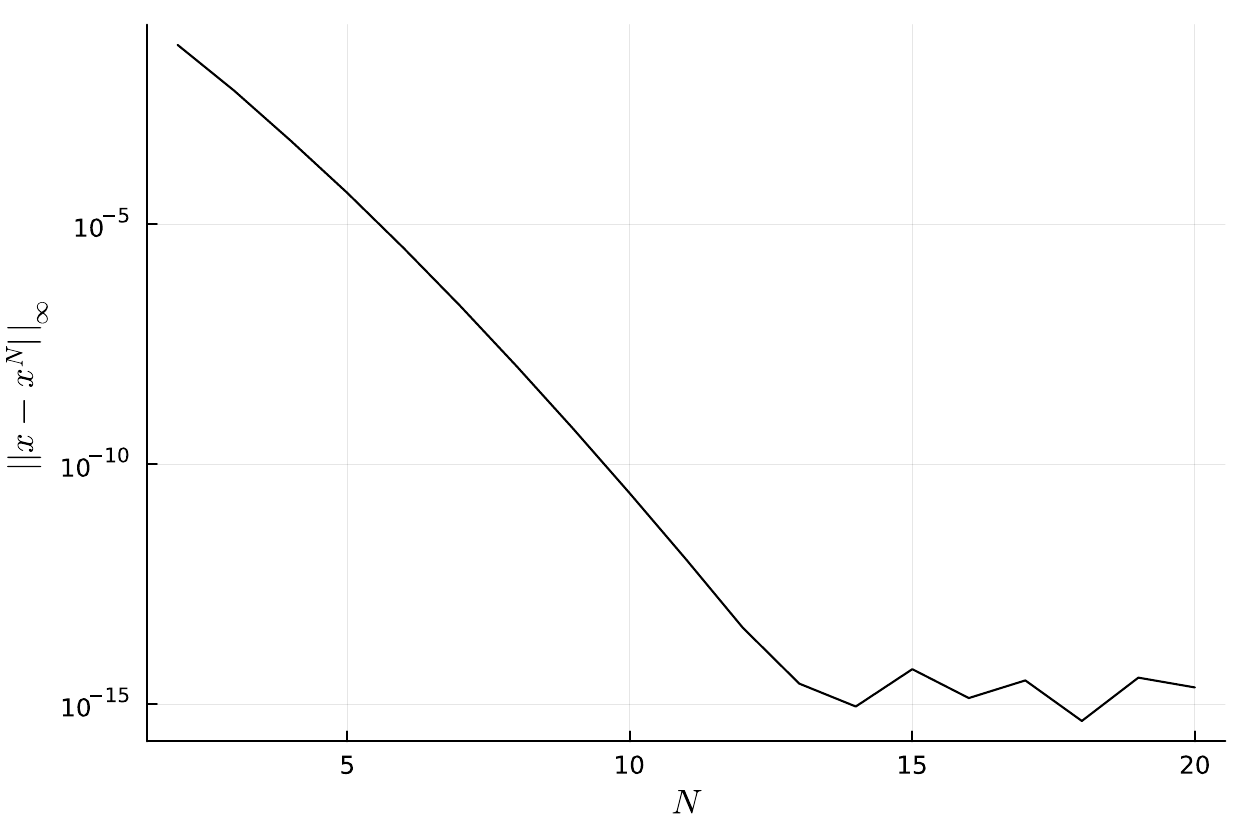}
	\includegraphics[scale = 0.42]{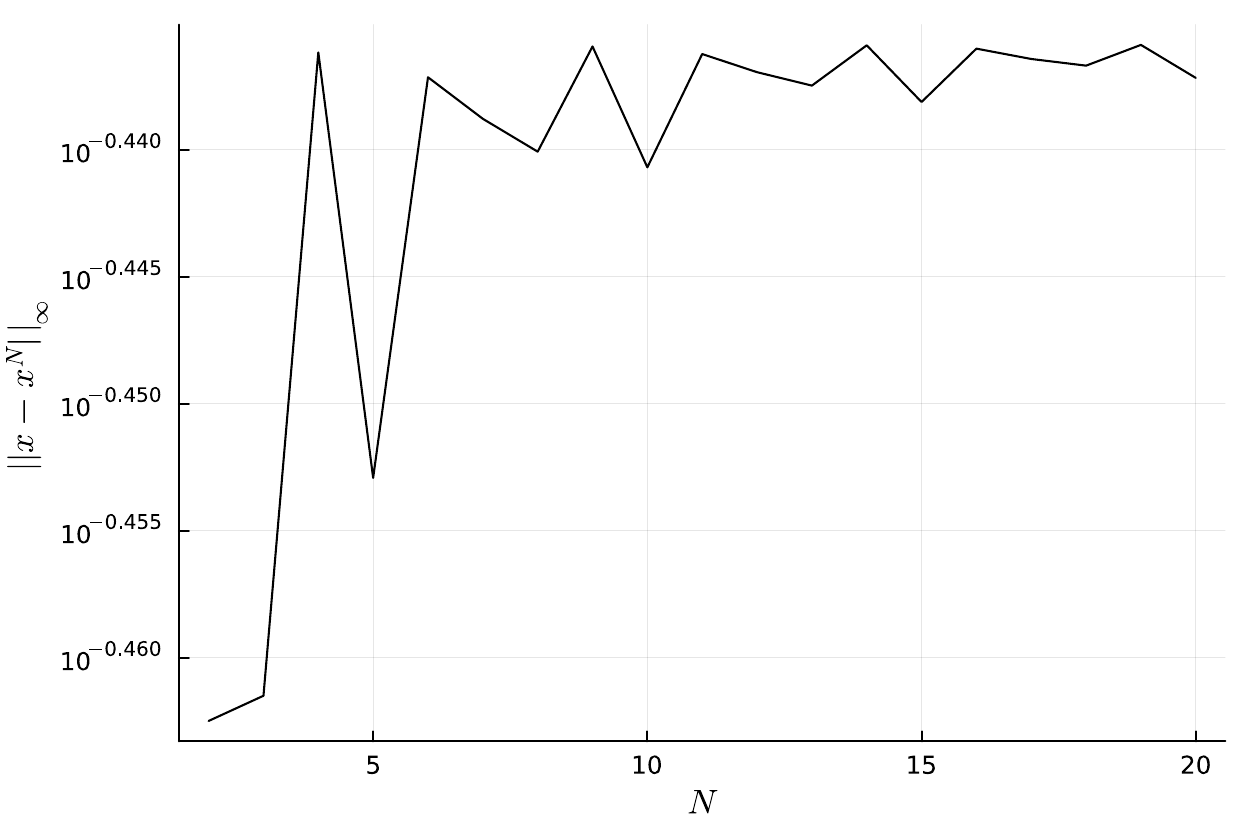}
	\caption{$L^\infty$ error of the spectral collocation scheme applied to \eqref{eqn:NumericalSmooth}. On the left: smooth solution \eqref{eqn:NumericalSmooth}, on the right: logarithmic solution \eqref{eqn:NumericalLogarithm}. }
	\label{fig:ErrorSmooth}
\end{figure}

\section{Conclusion}
We have introduced and analyzed a class of non-linear Volterra equations generalizing capillary rise models with only continuity hypotheses on the kernel and non-linearity. Using spaces with prescribed modulus of continuity and integral type H\"older spaces, we proved the existence of solutions in $J_{\alpha,\beta}[0,1]$, derived the first sharp interpolation error bounds in these norms, and established convergence of a piecewise‐linear collocation scheme.  For smoother cases, our Legendre-node spectral collocation method achieves exponential convergence. Numerical experiments validate our theoretical findings and demonstrate robustness to low regularity.

Future research includes developing a posteriori estimators and adaptive mesh refinement in the \(J_{\alpha,\beta}\) setting or in the variable H\"older spaces.  Additional directions are the integration of fractional or stochastic formulations and the design of fast hierarchical solvers to handle weak singularities with near‐linear complexity. Also, we would like to consider parallel in time integration to make use of the multithreaded possibilities of modern computers. 

\section*{Acknowledgment}
Ł.P. is supported by the Polish National Agency for Academic Exchange (NAWA) under the Bekker Programme with the signature BPN/BEK/2024/1/00002. J.C. and K. S. are partially supported by the project PID2023-148028NB-I00.


\end{document}